\documentclass[12pt,leqno]{amsart}

\usepackage{amsmath,amssymb,amscd,amsthm}
\usepackage{latexsym}
\usepackage[dvips]{graphicx,epsfig}
\usepackage{graphicx}
\usepackage{color}
\usepackage{float}
\usepackage{booktabs}
\usepackage{multicol}
\newtheorem{theorem}{Theorem}
\newtheorem{lemma}{Lemma}
\newtheorem{proposition}{Proposition}

\newtheorem{corollary}{Corollary}

\newtheorem{claim}{Claim}

\newcommand{\f}[2]{\frac{#1}{#2}}

\newcommand{\ga}{\gamma}
\newcommand{\Ga}{\Gamma}

\newcommand{\beq}{\begin{equation}}
\newcommand{\eeq}{\end{equation}}
\newcommand{\beqna}{\begin{eqnarray*}}
\newcommand{\eeqna}{\end{eqnarray*}}
\newcommand{\beqn}{\begin{equation*}}
\newcommand{\eeqn}{\end{equation*}}
\newcommand{\bp}{\begin{proof}}
\newcommand{\ep}{\end{proof}}
\newcommand{\bprop}{\begin{proposition}}
\newcommand{\eprop}{\end{proposition}}
\newcommand{\bt}{\begin{theorem}}
\newcommand{\et}{\end{theorem}}
\newcommand{\bex}{\begin{Example}}
\newcommand{\eex}{\end{Example}}
\newcommand{\bc}{\begin{corollary}}
\newcommand{\ec}{\end{corollary}}
\newcommand{\bcl}{\begin{claim}}
\newcommand{\ecl}{\end{claim}}
\newcommand{\bl}{\begin{lemma}}
\newcommand{\el}{\end{lemma}}

\begin{document}

\title{A class of Finite difference Methods for solving inhomogeneous damped wave equations}
 \author{Fazel Hadadifard}
 \author{Satbir Malhi}
 \author{Zhengyi Xiao}
 \address{Fazel Hadadifard, Department of Mathematics, Drexel University
 }
 \email{fh352@drexel.edu}
 
 \address{Satbir Malhi, Department of Mathematics, Saint Mary's College of California}
 \email{smalhi@stmarys-ca.edu}

 \address{Zhengyi Xiao, Department of Mathematics, Franklin \& Marshall College}
\email{zxiao@fandm.edu }

\thanks{ }

\date{\today}

\subjclass[2000]{65M06, 37N30, 65N22 }

\keywords{damped wave equation, numerical method, Pad\'e approximation, compact finite difference scheme, unconditionally stable, convergence}
\begin{abstract}
In this paper, a class of finite difference numerical techniques is presented to solve the second-order linear inhomogeneous damped wave equation. The consistency, stability, and convergences of these numerical schemes are discussed. The results obtained are compared to the exact solution, ordinary explicit, implicit finite difference methods, and the fourth-order compact method (FOCM).
%{\color{blue}by Hussain et al. [Pakistan Journal of Science, 64(2):122, 2012]}. T
The general idea of these methods is developed by using $C_0$-semigroups operator theory. We also showed that the stability region for the explicit finite difference scheme depends on the damping coefficient.

\end{abstract}

\maketitle

 \section{introduction}
The damped wave equation is an important evolution model. Physicists and engineers widely use it in describing the propagation of water waves, sound waves, electromagnetic waves, etc.  For instance, a model that describes the transverse  vibrations of a string of a finite length in the  presence of an external force proportional to the velocity satisfies the following partial differential equation (PDE)
 \begin{eqnarray}\label{eq1}
 	u_{tt}=\Delta u-\gamma(x)u_t+g(x,t),  \hskip 10pt\text{for} \hskip 10pt~~~a\leq x\leq b, ~~t\in R,
 \end{eqnarray}
 with initial conditions
 \[
 u(x,0)=\phi(x), ~~~~~~~~~u_t(0,x)=\psi(x), \hskip 10pt\text{for}~ a\leq x \leq b,
 \]
 and boundary conditions
\[
	u(a,t)=u_a(t) \hskip 10pt u(b,t)=u_b(t), \hskip 30pt~~~t\in R,
\]
where $\gamma\geq 0$ is the damping force, $u(x, t)$  is the  position of a point $x$ in the string, at instant $t$. The functions $\phi(x), \psi(x)$  and their derivatives are continuous functions of $x$ and the forcing function $g(x,t)\in L_x^1(\mathbb{R})$. The study of the numerical solution of this model will be our main focus in this article. 

In general, the damping reduces the amplitude of vibration, and therefore, it is desirable to have some amount of damping to achieve stability in the system.  One can find a detailed study  in \cite{evans10,burq2016exponential,rauch1974exponential} of the effect of damping in the long-time stability of the equation \eqref{eq1}.  Also, for practical purposes, it is important to know how much damping is needed in the system to ensure the fastest decay rate in the amplitude of the wave as time evolves. For example, in the case of the 3D tsunami wave, we would like to know the size and structure of the damping force to bring the amplitude of the tsunami to a safe level before it hits the shore (see \cite{segur2007waves} and references within). In the case of the damping terms as a function of time and space, obtaining an analytic solution is a challenging problem. There comes the numerical study to find the approximate solution to such problems.   
 In recent years, much attention has been given to studying the behaviours of the numerical solution of \eqref{eq1}; see for example \cite{smith1985numerical,christie1976finite,ozer2006one,gao2007unconditionally,larsson1991finite}.
 
 In this manuscript,  we develop a class of methods based on the properties of $C_0$-semigroups of the evolution equations, as well as the finite difference method (FD). Generally speaking, the FD methods are easy to apply to partial differential equitations, but they may not lead to optimal results depending on the type of equation. The techniques used in this article take advantage of the $C_0$-semigroup property and Pad\'e approximation, which lead to a better performance of new numerical schemes presented in this article.  
 	
At the time of writing this paper, we became aware of \cite{mohebbi2013fourth} that have a similar approach in which the authors drive a fourth-order implicit finite difference scheme to solve a second-order telegraph equation with constant coefficients. However, the author of \cite{mohebbi2013fourth}  did not consider the explicit finite difference schemes and used the higher-order approximation terms of the space derivative and time integration to attain higher-order accuracy of the numerical solution. In this manuscript, in addition to driving  a class of explicit and implicit methods, we discussed the issue of the instability of the explicit finite difference methods. Moreover, this paper explains the importance of  the non-zero damping term in the existence of the stability region as well. We have also shown that the explicit finite difference method produces  better results and costs a lot fewer calculations  in its stability region. 
%It would be an interesting problem to find the stability region of the explicit scheme in the set-up of \cite{mohebbi2013fourth}. 
 
An outline of the contents of this paper is as follows. In section \ref{sec:0}, we set our numerical schemes and derive our method. Section \ref{sec:1} is devoted to the analytical properties of the method, i.e., consistency, stability, and convergence. Finally, in section \ref{sec:2}, the numerical results of our method are compared  with some of the existing methods.

 \section{The semigroup approach}\label{sec:0}
 {\color{black} To present a more convenient form of \eqref{eq1}, we define a new vector  function
 \begin{eqnarray}
U(x, t)=(u,u_t)^{T},\ \ U_0=(\phi(x), \psi(x))^{T}.
 \end{eqnarray}
With these changes, the equation \eqref{eq1} turns into an evolution equation of first-order in time 

 \begin{eqnarray} \label{eqn2}
 U_t=\mathcal{A}U+G,
 \end{eqnarray}
where
  \begin{eqnarray*} 
 \mathcal{A}=\left(\begin{array}{cc}
 0&I\\\\\ \Delta&-\gamma(x)\end{array}\right),\ \ \ \ G(x,t)=\left(\begin{array}{c}
 0\\\\g(x,t)
 \end{array}\right),
 \end{eqnarray*}
 with initial condition
 \[
 U(x, 0)=\left(u(x,0),u_t(x,0)\right)^{T}.
 \] 
}

 The system above is defined on a   Hilbert space $\mathcal{H}=H^1[a,b]\times L^2(\mathbb{R})$.
 The  domain of $\mathcal{A}$ is  $D(\mathcal{A})=H^2[a,b]\times H^1(\mathbb{R})$. Since $-\mathcal{A}$ is a dissipative and invertible operator on a Hilbert space, it generates a $C_0$-semigroup of contractions for $t\geq 0$ by the Lumber-Phillips theorem \cite{lumer1961dissipative}. Also, note that the inclusion $D(\mathcal{A})\hookrightarrow\mathcal{H}$ is compact by the Rellich-Kondrachiv theorem. Thus, the spectrum of $\mathcal{A}$ contains only eigenvalues of finite multiplicity.
 \subsection{Discretization}
 We use the central discretization for the Laplacian operator $\Delta$ as
 \begin{eqnarray*}
 \Delta u(x,t)=\frac{u(x-h,t)-2u(x,t)+u(x+h,t)}{h^2}.
 \end{eqnarray*}

 We set the mesh points 
\[
x_i=a+ih,~~i= 0,1,2\dots N,~~~~~~~~~\text{where $h=\frac{b-a}{N}$}
\]
of the interval $[a,b]$. Then the continuous operator $\mathcal{A}$ can be approximated  by the  matrix operator
\[
\mathcal{M}_{(2N-2)}=\left[\begin{array}{cc}
0&I\\\\\frac{1}{h^2}A&-\Gamma
\end{array}
\right],
\]
where $I$ is the identity matrix of order $N-1$, and % $A$ is tridiagonal matrix of order $N-1$ with main diagonal elements $-2$ and first diagonal below  and the first diagonal above are all 1, $\Gamma$ is a diagonal matrix with diagonal elements $\gamma(x_n)$. 
\vskip 5pt
{\tiny
\begin{eqnarray}\label{eqnn4}
A=\left[\begin{array}{cccccc}
-2&1&0&\cdots&0\\\\
1&-2&1&\cdots&0\\ 
0&1&-2&\ddots&0\\
\vdots& \ddots & \ddots &\ddots&\vdots\\ 
%0&\cdots&0&1&-2&1\\
0&\cdots&0&1&-2\\
 \end{array}
\right]_{(N-1)\times (N-1)}
,\Ga=\left[\begin{array}{cccc}
\ga(x_1)&0&\cdots&0\\ \\
0&\ga(x_2)&\dots&0\\ \\
0&0&\cdots&0\\ 
\vdots& \ddots & \ddots &\vdots\\ 
%0&\cdots&\ga(x_{N-2})&0\\

0&\cdots&0&\ga(x_{N-1})\\

\end{array}
\right]_{(N-1)\times (N-1).}
\end{eqnarray}}

The discrete operator {$\mathcal{M}_{(2N-2)}$} is defined on the finite-dimensional Banach space $X_{(2N-2)}=\mathbb{C}^{(2N-2)}[a,b]$. \\
Let $V_{2N-2}(t)=\left[\begin{array}{c}
u(x_1,t), u(x_2,t)\dots u(x_{N-1},t), u_t(x_1,t), \cdots u_t(x_{N-1})
		\end{array}
		\right]^T $ be a vector which discretizes  the function $U(x, t)= (u(x, t), \partial_t u(x, t))$ over the interval $[x_1,x_{N-1}]$, then \eqref{eqn2} leads us to the following dynamical system
\begin{eqnarray}\label{eq3}
	\frac{dV_{2N-2}(t)}{dt}=\left[\begin{array}{cc}
		0&I\\\\\frac{1}{h^2}{A}&-\Gamma
		\end{array}
	\right]V_{2N-2}(t)+\left[\begin{array}{c} 0\\\\G(t)\end{array}\right]+\left[\begin{array}{c} 0\\\\\frac{1}{h^2}B(t)\end{array}\right],
	\end{eqnarray}
	where $G(t)=\left[ g(x_1,t), g(x_2,t),\dots, g(x_{N-1},t)\right]^T$, $B(t)=\left[\begin{array}{c}
		u_a(t), 0, 0, \hdots, 	0, 0, 	u_b(t)	\end{array}\right]$  and the initial condition 
\begin{eqnarray*}
	V_{2N-2}(0)=\left[\begin{array}{c}
		\phi(x_1), \phi(x_2)\dots \phi(x_{N-1}),\psi(x_1), \cdots \psi(x_{N-1})
	\end{array}
	\right]^T.
\end{eqnarray*}
We will now drop the subscript $2N-2$  and write $V_{2N-2}(x,t)$ by $V(t)$, and $\mathcal{M}_{2N-2}$ by $\mathcal{M}$ in the rest of our presentation. 

Since $\mathcal{M}$ is a bounded linear operator on a finite-dimensional space $X_{(2N-2)}\times H_0^1(\mathbb{R})$, it generates a $C_0$-semigroup for each $N$. Then, by using the $C_0$-semigroup theory of inhomogeneous evolution equations, we can construct the sequences of approximating solutions to \eqref{eq3} as
\begin{eqnarray*}
	V(t)=e^{\mathcal{M}t}V(0)+\int_{0}^{t}e^{\mathcal{M}(t-s)}F(s)~ds,
	\end{eqnarray*}
where 
\begin{eqnarray*}
	F(t)=\left[\begin{array}{c} 0\\\\G(t)\end{array}\right]+\left[\begin{array}{c} 0\\\\\frac{1}{h^2}B(t)\end{array}\right].
	\end{eqnarray*}
We replace $t$ by $t+k$ in the above equation and use the $C_0$-semigroup property, $e^{\mathcal{M}(t+k)}=e^{\mathcal{M}t}e^{\mathcal{M}k}$, we get
\begin{eqnarray*}\label{eqn4}
	V(t+k)&=&e^{\mathcal{M}(t+k)}V(0)+\int_{0}^{t+k}e^{\mathcal{M}(t+k-s)}F(s)~ds\\
	&=&e^{\mathcal{M} k}e^{\mathcal{M} t}V(0)+e^{\mathcal{M}k}\int_{0}^{t}e^{\mathcal{M}(t-s)}F(s)~ds+e^{\mathcal{M} k}\int_{t}^{t+k}e^{\mathcal{M}(t-s)}F(s)~ds\nonumber\\
		&=&e^{\mathcal{M} k}\left(V(t)-\int_{0}^{t}e^{\mathcal{M}(t-s)}F(s)~ds\right)+e^{\mathcal{M}k}\int_{0}^{t}e^{\mathcal{\mathcal{M}}(t-s)}F(s)~ds\\&+&e^{\mathcal{M}k}\int_{t}^{t+k}e^{\mathcal{M}(t-s)}F(s)~ds.
\end{eqnarray*}
Thus, 
\begin{eqnarray}\label{eqn7}
	V(t+k)&=&e^{\mathcal{M}(k)}V(t)+e^{\mathcal{M}k}\int_{t}^{t+k}e^{\mathcal{M}(t-s)}F(s)~ds.
\end{eqnarray}

To approximate the term $e^{\mathcal{M}k}$, we {\color{black}make use of the rational approximation of exponential  functions, i.e., the Pad\'{e} approximation. 
\subsection{Pad\'{e} Approximant}  The Pad\'{e} approximation is  {\color{black} a rational } approximation of a function of a given order \cite{baker1981pade}. The technique was developed around 1890 by Henri Pad\'{e}, but it goes back to  F. G. Frobenius who introduced the idea and investigated the features of rational approximations of power series. The Pad\'{e} approximation is usually superior when functions contain poles because the use of rational function allows them to be well represented. The Pad\'{e} approximation often gives a better approximation of the function than truncating its Taylor series, and it may still work where the Taylor series does not converge.

Pad\'{e} approximation gives the exponential functions  $e^\theta$ as
\begin{eqnarray*}
	e^\theta=\frac{1+a_1\theta+a_2\theta^2+\cdots+a_T\theta^T}{1+b_1\theta+b_2\theta^2+\cdots+_T\theta^S}
	+c_{S+T+1}\theta^{S+T+1}+O(\theta^{S+T+2}),
	\end{eqnarray*} 
where $C_{S+T+1}$, $a_i$'s and $b$'s are constants. The rational function 
\begin{eqnarray}\label{eqn5}
	R_{S,T}(\theta):=\frac{1+a_1\theta+a_2\theta^2+\cdots+a_T\theta^T}{1+b_1\theta+b_2\theta^2+\cdots+_T\theta^S}=\frac{P_T(\theta)}{Q_S(\theta)}
	\end{eqnarray} 
is the so-called Pad\'{e} approximation of order $(S,T)$ to $e^\theta$ with the leading error $c_{S+T+1}\theta^{S+T+1}$.
	The table below gives some Pad\'{e} approximations of the exponential function\cite{smith1985numerical}.
	
\[\begin{array}{p{2in}p{3in} p{2in} }
	\toprule
	\vskip 2pt
	\centering(S,T)&	\vskip 2pt\centering$R_{S,T}(\theta)$&	\vskip 2pt$\text{Leading error}$\vskip3pt\\
	\midrule
	\centering\vskip2pt(0,1)&\centering\vskip2pt$1+\theta$&\vskip2pt$\frac{1}{2}\theta^2$\vskip3pt\\

	\centering\vskip2pt(0,2)&\centering\vskip2pt$1+\theta+\frac{1}{2}\theta^2$&\vskip2pt$\frac{1}{6}\theta^3$\vskip3pt\\

		\centering\vskip2pt(1,0)&\centering\vskip2pt$\dfrac{1}{1-\theta}$&\vskip2pt$-\frac{1}{2}\theta^2$\vskip3pt\\

		\centering\vskip2pt(1,1)&\centering\vskip2pt$\dfrac{1+\frac{1}{2}\theta}{1-\f{1}{2}\theta}$&\vskip2pt$-\frac{1}{12}\theta^3$\vskip3pt\\
	\bottomrule
	\end{array}
		\]
Now combining \eqref{eqn4} and \eqref{eqn5}, we get
\begin{eqnarray}\label{eqn61}
	Q_S(\mathcal{M}k)	V(t+k)
	&=&P_T(\mathcal{M}k)V(t)\\&+&P_T({\mathcal{M}k})\int_{t}^{t+k}P_T({\mathcal{M}(t-s)})(Q_S({\mathcal{M}(t-s)}))^{-1}F(s)~ds.\nonumber
\end{eqnarray}
For the integration term on the right-hand side, one can use the numerical integration formula. Here, we will use the Trapezoidal approximation of integration to  get the following numerical scheme
	\begin{eqnarray}\label{eqn6}
Q_S(\mathcal{M}k)	V(t+k)
	&=&P_T(\mathcal{M}k)V(t)+\frac{k}{2}P_T(\mathcal{M}k)F(t)+\frac{k}{2}Q_S(\mathcal{M}k)F(t+k).
	\end{eqnarray}
This is the general form of our scheme, and each choice of $Q_S$ and $P_T$  produces  explicit and implicit finite difference  methods to the solution of the damped wave equation \eqref{eq1}. { Next,  we present two schemes by taking $(S, T)= (0, 1)$  and $(S, T)= (1, 1)$. Similarly, we can develop more schemes of different order by taking different values of $S$ and $T$ mentioned in the table above.}
\\
{\bf Explicit Method ($FD-(0,1)$)}: If we set $(S, T)= (0, 1)$ i.e. $Q_0(\theta)=1$ and $P_1(\theta)=1+\theta$ in \eqref{eqn6}, we will obtain the FD-(0,1)  as
{\color{black}
\begin{eqnarray}\label{6}
 \begin{cases}
	V{(t+k)}=(1+\mathcal{M}k)V{(t)}+\frac{k}{2}(I+\mathcal{M}k)F(t)+\frac{k}{2}F(t+k),\\
	V^0= [u_1(0), \cdots, u_{N-1}(0), \partial_t u_1(0), \cdots, \partial_t u_{N-1}(0)].
	 \end{cases}
\end{eqnarray}
{\bf Implicit Method ($FD-(1,1)$)}:  By a choice of  $P_1(\theta)=1+\frac{1}{2}\theta$ and  $Q_1(\theta)=1-\frac{1}{2}\theta$ in \eqref{eqn6}, we will obtain the FD-(1,1) as
{\color{black}
 \begin{eqnarray}\label{eqn7.1}
 \begin{cases}
\left(1-\frac{1}{2}\mathcal{M}k\right)V(t+k)=\left(1+\frac{1}{2}\mathcal{M}k\right)V(t)+\frac{k}{2}\left(I+\frac{1}{2}\mathcal{M}k\right)F(t)\\
\hspace{8em} + \frac{k}{2}\left(I-\frac{1}{2}\mathcal{M}k\right)F(t+k),\\
	V^0= [u_1(0), \cdots, u_{N-1}(0), \partial_t u_1(0), \cdots, \partial_t u_{N-1}(0)].
\end{cases}
\end{eqnarray}

In the case of the implicit method, we need to solve a more extensive system of equations in each time step due to the implicit nature of the system}. However,  the analysis and numerical results suggest that the implicit scheme gives us an accurate approximation and, more importantly, an unconditionally stable scheme.

\section{Consistency, Stability and Convergence}\label{sec:1}
In this section, we will investigate the analytical properties of our numerical schemes \eqref{6} and \eqref{eqn7.1}. We will prove that the numerical methods \eqref{6} and \eqref{eqn7.1} are consistent, stable,  and hence convergent. We will use the direct analysis to prove the consistency, the matrix method  to prove the stability,
 and the Lax-equivalence theorem to prove the convergence of our numerical schemes. 
\subsection{Consistency}
Given a partial differential equation $L u = f$ and a finite difference scheme,
$F_{h,k}v = f$, we say that the finite difference scheme is consistent with the
partial differential equation if for any smooth function $\phi(x,t)$,
\[
L\phi-F_{h,k}\phi\to0~~\text{as}~~h,k\to0,
\]
or in other words, the local truncation goes to zero as the mesh size $h$ and $k$  tends to zero. 
The  partial differential equation
	\begin{eqnarray*} 
U_t-\left(\begin{array}{cc}
0&I\\\\\ \Delta&-\gamma(x)\end{array}\right)U-\left(\begin{array}{c}
0\\\\g(x,t)
\end{array}\right)=0,
\end{eqnarray*}
 is approximated at the point $(x_i,t)$ by the $n^{th}$ row of the following difference equations
\begin{eqnarray*}\frac{1}{k}\left(Q_S({\mathcal{M}k})V(t+k)-P_T({\mathcal{M}k}) V(t)\right)-\frac{1}{2}P_T({\mathcal{M}k})F(t)-\frac{1}{2}Q_S({\mathcal{M}k})F(t+k)=0\nonumber,
\end{eqnarray*}
for $n=1,2, \cdots, (2N-2).$\\
Then the local truncation error $T_{i, t}(U)$ is defined as the $n^{th}$ row of\begin{eqnarray*}
 \frac{1}{k}\left(Q_S({\mathcal{M}k})U(t+k)-P_T({\mathcal{M}k})U(t)\right)-\frac{1}{2}P_T({\mathcal{M}k})F(t)-\frac{1}{2}Q_S({\mathcal{M}k})F(t+k),
\end{eqnarray*}
for $n=1, \cdots, (2N-2)$.\\
The truncated error depends on the choice of $Q_S$ and $P_T$. Therefore, we should consider them case by case. Here we consider  $FD-(0, 1)$ and $FD-(1,1)$. The remaining cases follow  the same path. 
\subsubsection{$FD-(0,1)$} 
The local truncation error $T^{0,1}_{i, t}(U)$ of the explicit $FD-(0,1)$ is defined as the  $n^{th}$ row of 
\begin{eqnarray*}
		&& \frac{1}{k}\left(U(t+k)-(I+{\mathcal{M}k})U(t)\right)-\frac{1}{2}(I+{\mathcal{M}k})F(t)-\frac{1}{2}F(t+k)
			\end{eqnarray*}
	for 	$n= 1, 2\cdots, (2N-2)$.\\
	Thus for $i= 1, 2, \cdots N-1$, we get the following system of $(2N-2)$ equations
	\begin{eqnarray*}
		T^{0,1}_{i, t}(U)=	\frac{1}{k}\left(u(x_i, t+k)-u(x_i,t)\right)-u_t(x_i,t)-\frac{k}{2}g(x_i,t),
	\end{eqnarray*}
and 
\begin{eqnarray*}
	T^{0,1}_{i+N-1, t}(U)=\frac{1}{k}u_t(x_i,t+k)-\frac{1}{h^2}(u(x_i-h, t)-2u(x_i,t)+u(x_i+h,t))\\\\-{\frac{1}{k}(1-k\ga(x_i))u_t(x_i,t)}-~\frac{(1-\ga(x_i)k)}{2}g(x_i,t)-\frac{1}{2}g(x_i,t+k).
		\end{eqnarray*}
	By Taylor series expansion, we get 
	\begin{eqnarray*}
		T^{0,1}_{i, t}(U)=\frac{k}{2!}u_{tt}(x_i,t)+\frac{k^2}{3!}u_{ttt}(x_i,t)+\cdots-\frac{k}{2}g(x_i,t)
\end{eqnarray*}
and
\begin{eqnarray*}
			T^{0,1}_{i+N-1, t}(U)&=&	\left(u_{tt}(x_i, t)-u_{xx}(x_i, t)+\ga(x_i)u_t(x_i, t)-g(x_i, t)\right)\\\\
		&&+\frac{k}{2!}u_{ttt}(x_i, t)+O(k^2)-\frac{2h^2}{4!}u_{xxxx}(x_i, t)+O(h^4)+\frac{\ga(x_i)k}{2}g(x_i, t)\\\\
		&&-\frac{k}{2}g_{t}(x_i, t)+O(k^2).
\end{eqnarray*}
for~ $i=1,2,\cdots, (N-1)$.
\\
{By \eqref{eq1}, the last ($N-1$), equations can be written as }
\begin{eqnarray*}
	T^{0,1}_{i, t}(U)
&=&\frac{k}{2!}u_{ttt}(x_i, t)+O(k^2)-\frac{2h^2}{4!}u_{xxxx}(x_i, t)+O(h^4)+\frac{\ga(x_i)k}{2}g(x_i, t)\\\\
		&-&\frac{k}{2}g_{t}(x_i, t)+O(k^2).
\end{eqnarray*}
We observe as  $h$ and $k$ go to zero, the truncation error $T_{i, t}(U)\to0$. Hence, the numerical scheme is consistent. 
. 
\subsubsection{$FD-(1,1)$ } The local truncation error $T^{1,1}_{i, t}(U)$ of the explicit $FD-(1,1)$ is defined as the $n^{th}$ row of 
\begin{eqnarray*}
 \frac{1}{k}\left(\left(I-\frac{1}{2}{\mathcal{M}k}\right)U(t+k)-\left(I+\frac{1}{2}{\mathcal{M}k}\right)U(t)\right)-\frac{1}{2}\left(I+\frac{1}{2}{\mathcal{M}k}\right)F(t)-\frac{1}{2}\left(I-\frac{1}{2}{\mathcal{M}k}\right)F(t+k)
	\end{eqnarray*}
for $n=1,2,\cdots,(2N-2)$.\\
Thus for $i=1, 2\cdots, N-1$, we get the following system of $(2N-2)$ equations
	\begin{eqnarray*}
		T^{1,1}_{i, t}(U)=	\frac{1}{k}\left(u(x_i, t+k)-u(x_i, t)\right)-u_t(x_i, t)-\frac{k}{4}\left(g(x_i, t+k)-g(x_i, t)\right),
\end{eqnarray*}
and
{\begin{eqnarray*}
		T^{1,1}_{i+N-1, t}(U)&=&	\left[-\frac{1}{2h^2}\left(u(x_i-h, t+k)-2u(x_i, t+k)+u(x_i+h, t+k)\right)+\frac{1}{k}\left(1+\frac{\gamma k}{2}\right)u_t(x_i, t+k)\right]\\
&&	-\left[\frac{1}{2h^2}\left(u(x_i-h, t)-2u(x_i, t)+u(x_i+h, t)\right)+\frac{1}{k}\left(1-\frac{\gamma k}{2}\right)u_t(x_i, t)\right]\\
&&-\frac{1}{2}\left[(1+\frac{\ga k}{2})g(x_i, t)+(1-\frac{\ga k}{2})g(x_i, t+k)\right].
\end{eqnarray*}}
By Taylor series expansion, we get 
	\begin{eqnarray*}
	T^{1,1}_{i, t}(U)=	\frac{k}{2}u_{tt}(x_i, t)-\frac{k^2}{4}g_t(x_i, t)+O(k^3),
\end{eqnarray*}
and 
	\begin{eqnarray*}
	T^{1,1}_{i+N-1, t}(U)&=&	\left(u_{tt}(x_i, t)-u_{xx}(x_i, t)+\ga(x_i)u_t(x_i, t)-g(x_i, t)\right)\\
&+&	\frac{k}{2}u_{ttt}(x_i, t)+O(k^2)
		-\frac{k}{2}u_{xxt}(x_i, t)+{O}(k^2)
		-\frac{h^2}{2}u_{xxxx}(x_i, t)+O(h^4)\\
		&-&\frac{kh^2}{6}u_{xxxxt}+h^2O(k^2)
		-\frac{k}{2}g_t(x_i, t)+O(k^2)
		+\frac{k^3\ga(x_i)}{4}g_t(x_i, t)+O(k^3),
	\end{eqnarray*}
 for \text{ $i=1,2,\cdots, (N-1)$}.\\
 {By \eqref{eq1}, the last ($N-1$), equations can be written as }
 \begin{eqnarray*}
 	T^{1,1}_{i, t}(U)
 	&=&	\frac{k}{2}u_{ttt}(x_i, t)+O(k^2)
 	-\frac{k}{2}u_{xxt}(x_i, t)+{O}(k^2)
 	-\frac{h^2}{2}u_{xxxx}(x_i, t)+O(h^4)\\
 	&-&\frac{kh^2}{6}u_{xxxxt}+h^2O(k^2)
 	-\frac{k}{2}g_t(x_i, t)+O(k^2)
 	+\frac{k^3\ga(x_i)}{4}g_t(x_i, t)+O(k^3).
 \end{eqnarray*}
 As $h$ and $k$ go to zero, the truncation error $T_{i, t}(U)\to0$. Hence, the numerical scheme is consistent. 
\subsection{Stability} { To prove the stability of our numerical schemes, we show that there exists a region $\Lambda$ so that for every $h,k\in \Lambda$, all the eigenvalues of the amplification matrix related to the numerical schemes lie in or on the unit circle.}

%We define the error $e_n(t)$ function  as 
%\[
%e_n(t)=V(x_n,t)-U(x_n,t)
%\]
%where $V$ is the solution of our discrete solution and $U$ is the exact solution of equation (\ref{eqn2}). 
%\vskip 10pt
%We substitute $V(x_n,t)=e_n(t)+U(x_n,t)$ in  (\ref{eqn7}), we get 
%
%\begin{eqnarray*}
%e_n(t+k)+U(x_n,t+k)&=&e^{M(k)}\left(e_n(t)+U(x_n,t)\right)+\int_{0}^{k}e^{M(k-s)}F(t+s)~ds\\
%e_n(t+k)&=&e^{M(k)}\left(e_n(t)\right) -\left(U(t+k,t)-e^{{\color{red}\mathcal{M}}}U(x_n,t)\right)+\int_{0}^{k}e^{M(k-s)}F(t+s)~ds\\
%\end{eqnarray*}
%Sine the $U(x,t)$ is the exact solution, therefore it will also satisfies the our discrete version of the equation. This gives us 
%\begin{eqnarray*}
%e_n(t+k)&=&e^{M_n(k)}e_n(t)\\
%\end{eqnarray*}
%Thus, 
%\begin{eqnarray*}
%	||e_n(k)||&=&||e^{M_n(k)}||||e_n(0)||
%\end{eqnarray*}
%by spectral mapping theorem 
%\begin{eqnarray*}
%	||e_n(k)||&=&||e^{\lambda_nk}||||e_n(0)||
%\end{eqnarray*}
%Where $\lambda_n$ are the eigenvalues of $M$\\
%Thus the scheme is convergent as long as $e^{\lambda_nk}$ is bounded for $\forall n$ and for any fix $k$ since $e_n(0)\to 0$ as $n\to \infty$.

\begin{proposition}\label{thm1.1} The explicit FD-(0,1) approximation defined in \eqref{eqn6} is stable for  $k<\frac{2}{\gamma_*}$ and $\frac{\sqrt{k}}{h}<\frac{\sqrt{ \ga_*}}{2}$, where $\ga_*=\max_{x\in[a,b]}\gamma(x)$.
\end{proposition}
\noindent The following  lemma will be used to prove Proposition \ref{thm1.1}.
\begin{lemma}\label{lem1.1}
	Let $p(x)=ax^2+bx+c$ be a polynomial function with $a>0$, then necessary and sufficient conditions for the polynomial $p(x)$ to have the modulus of its roots less or equal to 1  are
	
	\begin{itemize}
		\item[(i)] $|c|<a$
		\item [(ii)] $p(1)>0$ and $p(-1)>0$.
			\end{itemize}
\end{lemma}
\noindent One can find the proof of the above lemma in \cite{jury1963roots,samuelson1941conditions}. \\
{\it \bf Proof of proposition \ref{thm1.1}.} 
The eigenvalues of the amplification matrix $I+k\mathcal{M}$  are the roots of the following quadratics equation
	\begin{eqnarray*}
		\lambda^2+(-2+\gamma(x_n)k)	\lambda+1-k\gamma(x_n)+4r^2\sin^2\left(\frac{n\pi}{2N}\right)=0,\ \ n=1, \cdots, (N-1),
	\end{eqnarray*}
	where $r=k/h$.\\
	Note for each $n$, there are two roots of the above polynomial, and hence we have $2N-2$ eigenvalues for the matrix $I+k
	\mathcal{M}$.
\\	
	Next, in order to satisfy the conditions $(i)$ and $(ii)$ of lemma \eqref{lem1.1}, we impose restrictions on $\ga_*$ and $r$. Indeed, the assumption $(i)$ is satisfied if
\begin{eqnarray*}
  -1<1-k\gamma(x_n)+4r^2\sin^2\left(\frac{n\pi}{2N}\right)<1,~~~n=1,2,\cdots, N-1.
  \end{eqnarray*}
The right-hand inequality gives us 
\begin{eqnarray*}
	&&4r^2\sin^2\left(\frac{n\pi}{2N}\right)<k\gamma(x_n)\leq k \ga_*,\\\\
	&&r^2<\frac{k \ga_*}{4\sin^2\left(\frac{n\pi}{2N}\right)}.
\end{eqnarray*}
Thus, $\frac{\sqrt{k}}{h}<\frac{\sqrt{\ga_*}}{2}$. 
\\
Now, the first part of the assumption (ii) is satisfied if 
\begin{eqnarray*}
	p(1)=4r^2\sin^2\left(\frac{n\pi}{2N}\right)>0,
\end{eqnarray*}
which is true as long as $r>0$. 
\\
Now, the second part of assumption (ii) is satisfied if 
\begin{eqnarray*}
	p(-1)=4-2k\gamma(x_n)+4r^2\sin^2\left(\frac{n\pi}{2N}\right)>0,
\end{eqnarray*}
which is true  if 
 \begin{eqnarray*}
k\gamma_* <2.
 \end{eqnarray*}
Hence  the second part of the assumption (ii) of lemma \eqref{lem1.1}  is satisfied for $k <\frac{2}{\ga_*}$.

Proposition \eqref{thm1.1} tells us that the damping term plays an important role in the stability of the explicit method \eqref{eqn6}. The finite difference scheme \eqref{eqn6} will be unstable for any values of $h$ and $k$ if the damping term $\gamma(x)$ is identically zero or $h$ and $k$  are out of the required bounds of  the proposition \eqref{thm1.1}.
%{\bfseries Note} Let $T_{max}$ be the final end point of time $t$. Then if we multiply the inequality $k <\frac{2}{\ga_*}$ by $T_{max}$, we get
%\[
%kT_{max} <\frac{2T_{max}}{\ga_*}
%\]
%This give us, the local existence of numerical schemes
%\[
%t <\frac{2T_{max}}{\ga_*}
%\] 
%This give us the maximum time 
\begin{proposition}
 The implicit FD-(1,1) approximation defined by \eqref{eqn7.1} is unconditionally stable.
\end{proposition}
\begin{proof} The eigenvalues of the matrix $\mathcal{M}$ are given by 
	\[
\lambda_n^{\pm}=-\frac{\gamma(x_n)}{2}	\pm\frac{1}{2}\sqrt{\gamma(x_n)^2- \frac{16}{h^2} \sin^2(\frac{n\pi}{2N})},~~~~~~n=1, \cdots, N-1.
	\]
	Then, by using  functional calculus, the eigenvalues $\mu_n^{\pm}$ of the matrix $(I-\frac{1}{2}k\mathcal{M})^{-1}((I+\frac{1}{2}k\mathcal{M}))$ are given by
		\[
\mu_n^{\pm}=	\frac{1+\frac{k}{2}\lambda_n^{\pm}}{1-\frac{k}{2}\lambda_n^{\pm}},\ \ n=1, \cdots,N-1.
	\]
	{Also, we have $Re(\lambda_n)\leq 0$ because $\gamma\geq 0$. Thus, for any values of $n, h, k $, and $\gamma(x_n)$, we get $|\mu_n^{\pm}|\leq 1$}. Hence, the implicit method \eqref{eqn7.1} is unconditionally stable.
	\end{proof}

%A finite difference scheme approximating a partial differential equation
%is a convergent scheme if for any solution to the partial differential equation, $u(t, x)$, and solutions to the finite difference scheme, $v_i^n$ such that $v_i^0$converges
%to $u_0(x)$ as $ih$ converges to $x$, then  $v(nk, ih)$ converges to $u(t, x)$ as $k$, $h$ converge to 0. 

%In order to present the convergence argument, we recall the following well-known Lax Equivalence theorem \cite{lax1956survey,richtmyer1994difference}. 
%\begin{theorem}{[The Lax Equivalence Theorem]}
%	Let $F_n$ be a consistent approximation to a well-posed linear initial-value (Cauchy)  problem. Then $F_n$ is convergent if and only if it is stable.  
%\end{theorem}
A direct application of the  Lax Equivalence Theorem  \cite{lax1956survey,richtmyer1994difference} leads to the convergence of our models. 
\begin{corollary}
	The finite difference explicit $FD-(0,1)$ of \eqref{eqn6} and implicit $FD-(1,1)$ of \eqref{eqn7.1} are convergent. 
\end{corollary}

\section{Performance of Numerical schemes}\label{sec:2}
In this section, we will see the  performance  of each finite difference scheme  on a sample problem.\\
{\bf Sample Problem:} We consider the following damped wave equation
\begin{eqnarray*}
	u_{tt}=u_{xx}-2u_t,
	\end{eqnarray*}
over the region $\Omega=[0\leq x\leq \pi]\times(t>0)$ with initial conditions 
\[
u(x,0)=\sin(x), \hskip 30pt u_t(x,0)=-\sin(x),
\]
and boundary conditions
\[
u(0,t)=0=u(\pi,t),\;\;\text{for}~  t>0.
\]
 The exact solution of the above problem is $u(x,t)=e^{-t}\sin(x)$.\\
%We chosen the following values of $h$ and $k$ to compute the approximation solution which we compare with analytical solution. 
%\vskip 10pt
%\begin{center}
%	\begin{tabular}{|p{1in} |  p{1in} | p{1in}|  p{1in} | }
%	\hline 
%	Trail &h&k&r\\
%	\hline 
%	1 &.08&.02&0.25\\
%	\hline 
%		2 &.04&.02&0.5\\
%	\hline 
%		3 &.027&.02&0.75\\
%	\hline 
%		4 &.0202&.02&.99\\
%	\hline 
%		4 &.02&.02&1\\
%	\hline 
%		4 &.018&.02&1.1\\
%	\hline 
%\end{tabular}
\begin{figure}[H]
	\includegraphics[width=.4\textwidth]{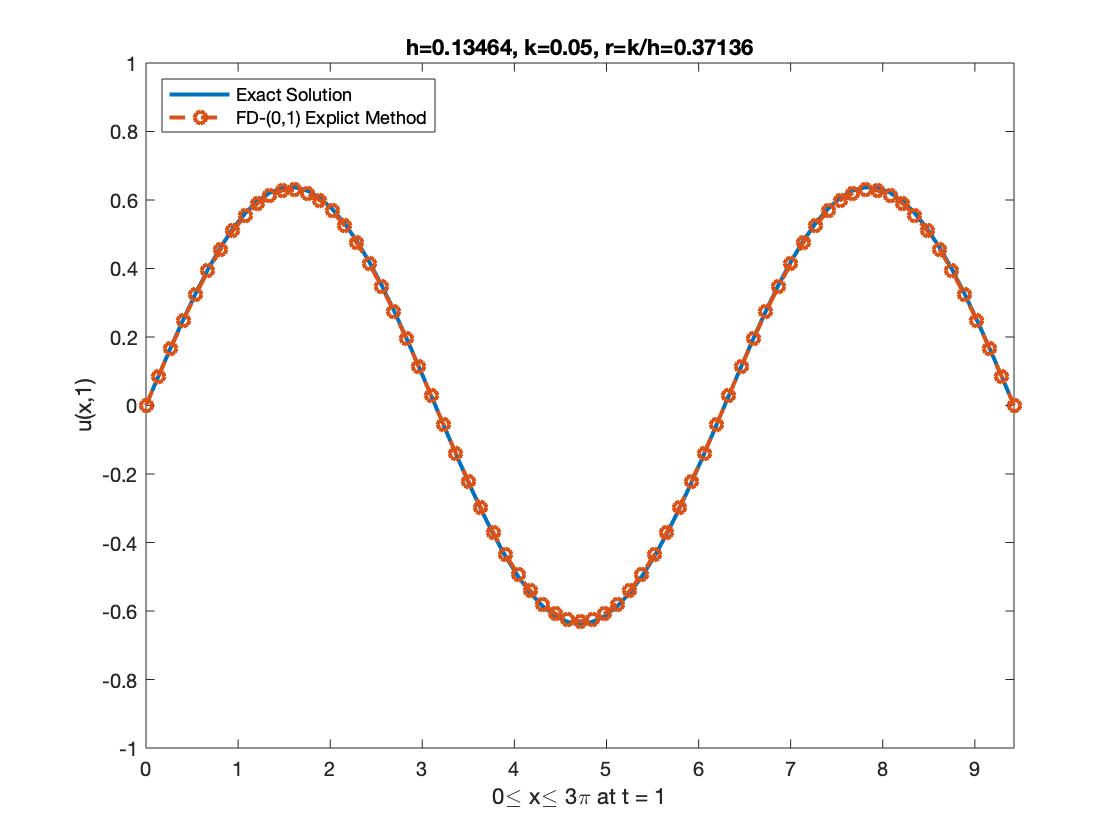}
	\caption{ The approximate solution given by explicit FD-(0,1) of \eqref{6} at $t=1$  with $k=0.05$, $h=0.13464$.}
	\label{fig1}
\end{figure}
\begin{figure}[H]
	\includegraphics[width=.4\textwidth]{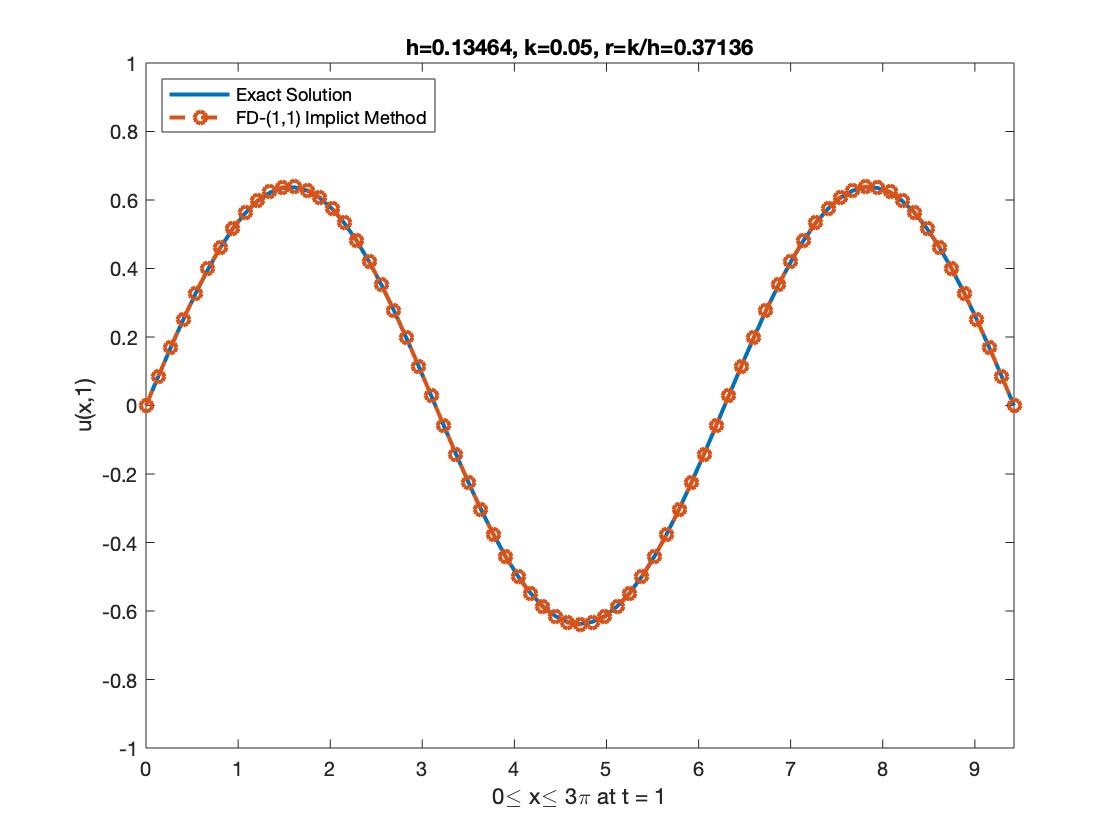}
	\caption{ The approximate solution given by implicit FD-(1,1) of \eqref{eqn7.1} at $t=1$  with $k=0.05$, $h=0.13464$. }
	\label{fig2}
\end{figure}
FIGURE \eqref{fig1} and \eqref{fig2} show the  numerical solutions using finite difference methods \eqref{eqn6} and \eqref{eqn7.1} at $t=1$. From the obtained numerical results, we can conclude that the numerical solutions are in good agreement with the exact solution. 
\subsection{Comparison with other methods}
In this section,  we  compare our result with the ordinary explicit and implicit finite difference methods mentioned  below. We also compare our result with the FOCM method of \cite{hussain2012fourth}. We take the same test example mentioned above for this comparison. \\
{ \bf Ordinary Explicit Finite Difference Scheme (OEFD):} The ordinary explicit finite difference scheme  in the matrix form is 
\begin{eqnarray}\label{eqn8}
(1+\frac{\gamma k}{2})u(t+k)=(2I-r^2A)u(t)+\left(\frac{\gamma k}{2}-1\right)u^{n-1}+r^2B(t),
\end{eqnarray}
where $r=k/h$,  $B(t)=\left[\begin{array}{c}
		u_a(t), 0, 0, \hdots, 	0, 0, 	u_b(t)	\end{array}\right]$, and the matrix $A$  is defined in equation \eqref{eqnn4}.\\
{\bf Ordinary Implicit Finite Difference Scheme(OIFD):} The ordinary implicit finite difference scheme  in the matrix form is

\begin{eqnarray}\label{eqn9}
\left(1+\frac{\gamma(x_n)k}{2}-\frac{r^2}{2}A\right)u(t+k)=\left(2+\frac{r^2}{2}A\right)u(t)+\left(\frac{\gamma(x_n)k}{2}-1\right)u(t-k)\\+\frac{r^2}{2}\left(B(t+k)+B(t)\right),\nonumber
\end{eqnarray}
where $r=k/h$,  $B(t)=\left[\begin{array}{c}
		u_a(t), 0, 0, \hdots, 	0, 0, 	u_b(t)	\end{array}\right]$, and the matrix $A$ is defined in equation \eqref{eqnn4}. The derivation of these  schemes can be found in \cite{lippold1980mitchell}.

FIGURE \eqref{fig31} and  \eqref{fig32} show the performances of our methods \eqref{6} and \eqref{eqn7.1} in comparison with  finite difference schemes \eqref{eqn8} and \eqref{eqn9} using $k=0.01$ and $h=0.063$. The  implicit FD-(1,1) produces a much better result even for a large value of $r$. When the values of $h$ and $k$  fail to satisfy the stability conditions of the explicit FD-(0,1), it can be seen that the numerical solution became unstable after some time iterations. However, it is interesting to see that even in this case the global numerical solution fails to exist, the local numerical solution does exist for a small time and { it was }very close to the exact solution. 
It is apparent that the explicit finite difference scheme \eqref{eqn8} and \eqref{eqn6} are not stable for large values of $r$. The implicit FD-(1,1) is very stable and produces a much better result { when compared} to the ordinary implicit finite difference scheme \eqref{eqn9}.
%\end{center}

\begin{multicols}{2}
\begin{figure}[H]
	\includegraphics[width=.4\textwidth]{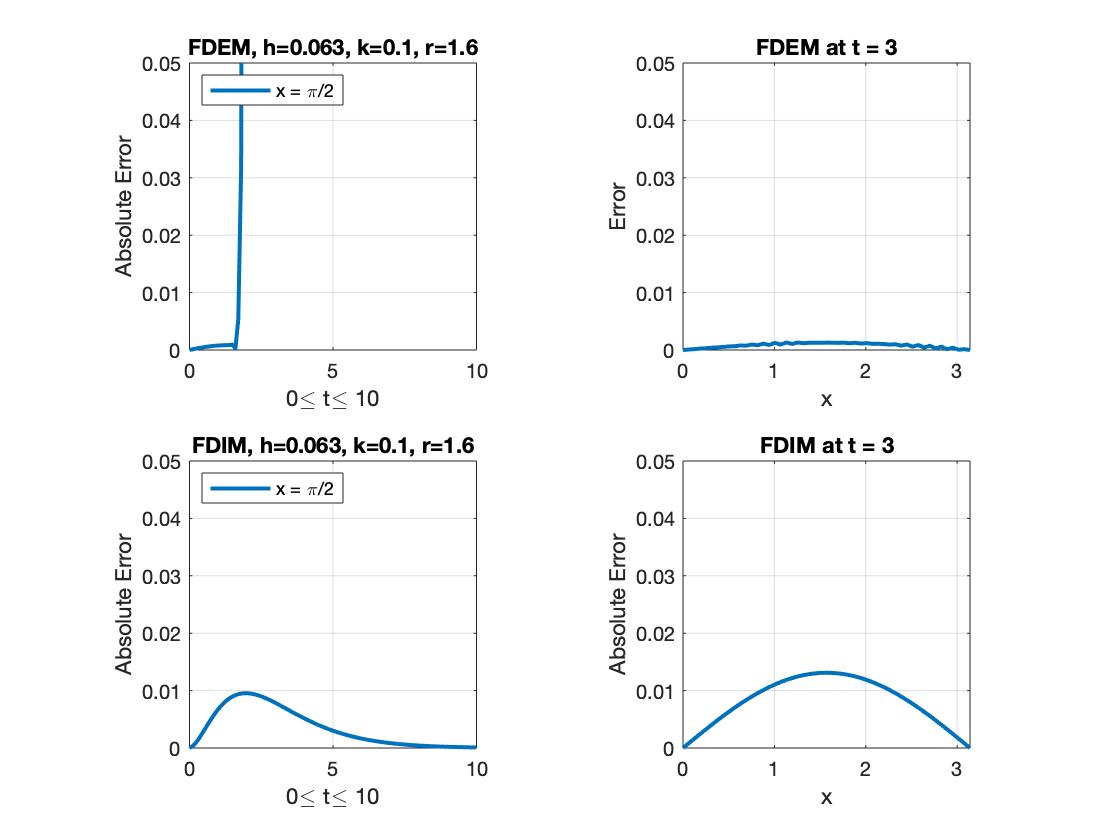}
	\caption{ The absolute error of the method \eqref{6} and \eqref{eqn7.1} for $r=1.5915$.}
	\label{fig31}
\end{figure}\begin{figure}[H]
\includegraphics[width=.4\textwidth]{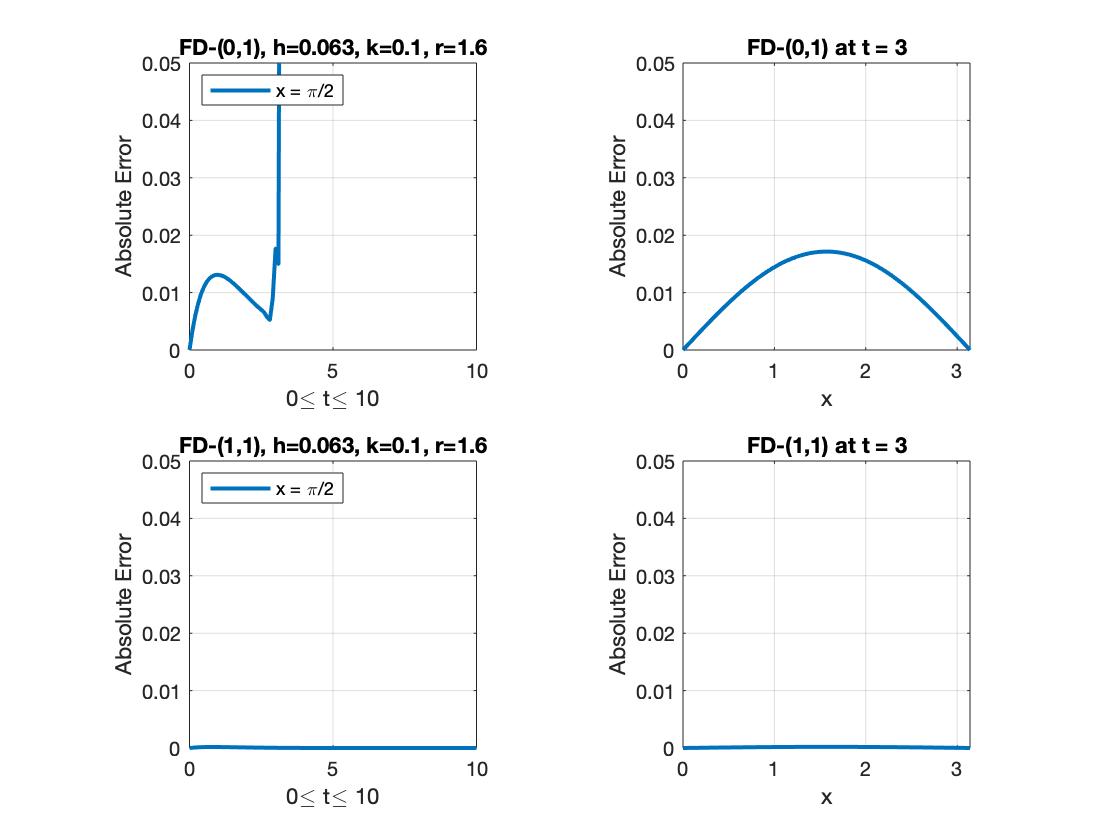}
\caption{ The absolute error of the method \eqref{eqn8} and \eqref{eqn9} for $r=1.5915$.}
\label{fig32}
\end{figure}
\end{multicols}

%\begin{figure}[H]
%		\includegraphics[width=.5\textwidth]{Fig31.jpg}
%	\includegraphics[width=.5\textwidth]{Fig32.jpg}
%		\caption{The absolute error comparison of method \eqref{6}, \eqref{eqn7.1}, \eqref{eqn8} and \eqref{eqn9}}
%	\label{fig3}
%\end{figure}

\begin{figure}[h!]
	\includegraphics[width=1\textwidth]{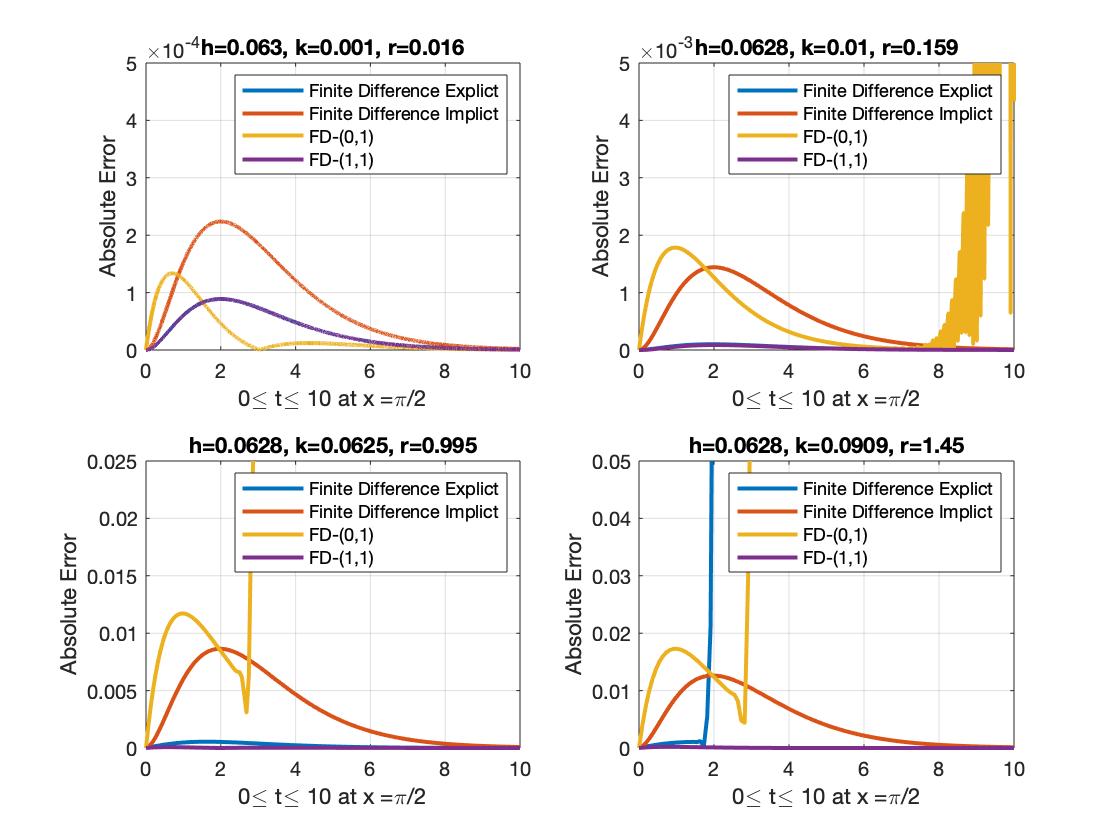}
	\caption{\tiny The absolute error  of the method \eqref{6}, \eqref{eqn7.1}, \eqref{eqn8} and \eqref{eqn9}.}
	\label{fig4}
\end{figure}

In the FIGURE \eqref{fig4}, we plotted the absolute error at the four different values of  $r=.016$,  $r=.159$, $r=.995$, and $r=1.45$. One can see for a small values of $r=0.016$, all the four schemes  produce fairly stable results. This shows that when our explicit finite difference FD-(0,1) satisfies the assumptions of proposition \eqref{thm1.1}, it is stable and produces better results than the other three. However, the performances of the explicit finite difference method \eqref{eqn6} and implicit finite difference FD-(1,1)  \eqref{eqn7.1} are very similar for  small values of $r$.
\begin{table}[h!]
	\centering

	\begin{tabular}{p{2.5cm}p{2.5cm} p{2.5cm} p{2.5cm} p{2.5cm} p{2.5cm} }
		\toprule
			\vskip 1pt	$x$ & 	\vskip 1pt FOCM&	\vskip 1pt OEFD  & 	\vskip 1pt OIFD 	&  	\vskip 1pt EX-(0,1)  & 	\vskip 1pt IM-(1,1)  \\[0.8ex]
		\midrule
			0&0&0	&0	&0&	0	\\[1ex]
		0.314159265&0.00012256&	6.29067E-05&	0.000135485	&0.001494844&	1.23932E-05
	\\[1ex]
		0.628318531&0.00022777&	0.000119656	&0.000257708&	0.002843363	&2.35734E-05
	\\[1ex]
		0.942477796&0.00031458&	0.000164692	&0.000354705&	0.003913553	&3.24459E-05
	\\[1ex]
		1.256637061	&0.00036955&0.000193607	&0.000416981&	0.004600658	&3.81425E-05
	\\[1ex]
		1.570796327	&0.00038865&0.00020357&	0.000438439	&0.004837418&	4.01054E-05
	\\[1ex]
		1.884955592	&0.00036955&0.000193607	&0.000416981	&0.004600658&	3.81425E-05
	\\[1ex]
		2.199114858&0.00031458&	0.000164692	&0.000354705&	0.003913553	&3.24459E-05
	\\[1ex]
		2.513274123	&0.00022777&0.000119656&	0.000257708&	0.002843363	&2.35734E-05
	\\[1ex]
		2.827433388	&0.00012256&6.29067E-05&	0.000135485&	0.001494844	&1.23932E-05
	\\[1ex]
		3.141592654&0&0&	0&	0	&0	\\[1ex]
			\bottomrule
		\end{tabular}
	\vskip 10pt
	\caption{ Absolute Error}
	\end{table}
% \begin{tabular}{p{2.5cm}p{2.5cm} p{2.5cm} p{2.5cm} p{2.5cm} p{2.5cm} }
% \toprule
% \vskip 1pt	$x$ & 	\vskip 1pt FOCM&	\vskip 1pt OEFD  & 	\vskip 1pt OIFD 	&  	\vskip 1pt EX-(0,1)  & 	\vskip 1pt IM-(1,1)  \\[0.8ex]
% \midrule
% \vskip 5pt	0&0	               &           	0	       &    	0             &	0	                 &	0	\\[1ex]
% 0.31416  &&	0.000141	   &	0.00039111	&	0.0026726	&	1.59E-06	\\[1ex]
% 0.62832& 	&	0.00026821	&	0.00074394	&	0.0050836 &	3.02E-06	\\[1ex]
% 0.94248&	&	0.00036915	&	0.001024	&	0.006997	 &	4.15E-06	\\[1ex]
% 1.2566	& &	0.00043397	&	0.0012037	&	0.0082255	&	4.88E-06	\\[1ex]
% 1.5708	&&	0.0004563	&	0.0012657	&	0.0086488	&	5.13E-06	\\[1ex]
% 1.885	& &	0.00043397	&	0.0012037	&	0.0082255	&	4.88E-06	\\[1ex]
% 2.1991	&&	0.00036915	&	0.001024	&	0.006997	&	4.15E-06	\\[1ex]
% 2.5133	& &	0.00026821	&	0.00074394	&	0.0050836	&	3.02E-06	\\[1ex]
% 2.8274	& &	0.000141	&	0.00039111	&	0.0026726	&	1.59E-06	\\[1ex]
% 3.1416	&0&	0	&	0	&	0	&	0	\\[1ex]
% \bottomrule
%\end{tabular}
%\vskip 10pt
%\caption{Absolute error}
%\end{table}

\begin{table}[h!]
	\centering

	\begin{tabular}[t]{p{2cm}p{3cm}p{3cm}p{3cm}p{3cm}}
		\toprule
	\vskip 1pt	$r$ & 	\vskip 1pt EFD  & 	\vskip 1pt IFD 	&	\vskip 1pt  EX-(0,1)  & 	\vskip 1pt IM-(1,1)  \\[1ex]
		\midrule
		1.59&	1.00967E+34	&0.002547509&	9.08234E+13&	2.231E-06
	\\[1ex]
		0.53&	3.05424E-05&	0.00079153&	2.18322E+11	&1.36036E-05
	\\[1ex]
		0.32&	2.04246E-05	&0.000473008&	3410.243641&	1.43835E-05
	\\[1ex]
		0.23&	1.76452E-05&	0.000339697&	0.011310925	&1.45754E-05
	\\[1ex]
		0.18&	1.64986E-05	&0.000266457&	7.84447E-05&	1.46457E-05	\\[1ex]
	\bottomrule
\end{tabular}
\vskip 10pt
	\caption{Maximum Error at $t=6$}
\end{table}
TABLE 1 shows the comparison between the errors generated by FOCM, OEFD, OIFD, $EX-(0,1)$ and $IM-(1,1)$ at $t=0.3$ with {\color{black}$h=\frac{\pi}{10}$ and $k=\frac{1}{10}$}.

TABLE 2 shows the magnitude of the maximum error at time $t = 6$ between the exact solution and the numerical solution obtained by using FOCM, OEFD, OIFD, $FD-(0,1)$, and $FD-(1,1)$ discussed above with different values of $h$ and $k$.

%The mesh point $h=\frac{\pi}{10}$ and $k=\frac{1}{10}$ are used to compute an approximation solution which we compare with analytical solution. %\begin{thebibliography}{100}

 \section{Conclusion}
 In this paper, a class of finite difference methods using the $C_0$-semigroup operator  theory  for solving the inhomogeneous damped  wave  equation is presented. The stability and consistency of the implicit and explicit methods are proved.
 Test examples are presented, and the results obtained are compared with the exact solutions. The comparison  certifies that implicit FD-(1,1) gives good results. Summarizing these results, we can say the general form of the new finite difference methods has a reasonable amount of calculations and the form is easy to use.
 All results are obtained by using MATLAB version 9.7.
 % \vskip 10pt
  %\noindent
 %{\bf \large Ethics declarations}
%\noindent {\bf Ethical statement:}
% \\

%  \noindent  {\bf Conflict of Interest:} The authors declare that they have no conflict of
    %	interest.
   \newpage
 \bibliographystyle{plain}
 % \bibliography{Reference}
% \end{thebibliography}

\end{document}